\documentclass[12pt,oneside,french,english]{amsart}
\usepackage[T1]{fontenc}
\usepackage[latin9]{inputenc}
\usepackage[letterpaper]{geometry}
\usepackage{amsthm}
\usepackage{amssymb}
\usepackage{esint}

\makeatletter

\newcommand{\noun}[1]{\textsc{#1}}

\numberwithin{equation}{section}
\numberwithin{figure}{section}
  \theoremstyle{plain}
  \newtheorem*{thm*}{Theorem}

\makeatother

\usepackage{babel}
\addto\extrasfrench{\providecommand{\fg}{\ifdim\lastskip>\z@\unskip\fi~\frqq}}

\begin{document}
\selectlanguage{french}%
\textbf{\hfill{}} \foreignlanguage{english}{\textit{Paru à Algebra
i Analiz, 23, 152-166, (2011).}}

\selectlanguage{english}%
\vspace{0.5cm}

\title{Asymptotic sharpness of a Bernstein-type inequality for rational
functions in $H^{2}$}

\author{Rachid Zarouf}
\begin{abstract}
A Bernstein-type inequality in the standard Hardy space $H^{2}$ of
the unit disc $\mathbb{D}=\{z\in\mathbb{C}:\,\vert z\vert<1\},$ for
rational functions in $\mathbb{D}$ having at most $n$ poles all
outside of $\frac{1}{r}\mathbb{D}$, $0<r<1$, is considered. The
asymptotic sharpness is shown as $n\rightarrow\infty$, for every
$r\in[0,\,1).$ 
\end{abstract}
\maketitle

\section*{I. Introduction}

First we recall the classical Bernstein inequality for polynomials:
we denote by $\mathcal{P}_{n}$ the class of all polynomials with
complex coefficients, of degree $n$: $P=\sum_{k=0}^{n}a_{k}z^{k}$.
Let \[
\left\Vert P\right\Vert _{2}=\frac{1}{\sqrt{2\pi}}\left(\int_{\mathbb{T}}\left|P(\zeta)\right|^{2}d\zeta\right)^{\frac{1}{2}}=\left(\sum_{k=0}^{n}\left|a_{k}\right|^{2}\right)^{\frac{1}{2}}.\]
 The classical inequality

\global\long\def\theequation{${1}$}
\begin{equation}
\left\Vert P'\right\Vert _{2}\leq n\left\Vert P\right\Vert _{2}\label{eq:}\end{equation}
 is known as Bernstein's inequality. A great number of refinements
and generalizations of (1) have been obtained. See {[}RaSc, Part III{]}
for an extensive study of that subject. The constant $n$ in (1) is
obviously sharp (take $P=z^{n}$).

Now let $\sigma=\left\{ \lambda_{1},\,...,\,\lambda_{n}\right\} $
be a sequence in the unit disc $\mathbb{D}$, the finite Blaschke
product $B_{\sigma}=\prod_{i=1}^{n}b_{\lambda_{i}}$, where $b_{\lambda}=\frac{\lambda-z}{1-\overline{\lambda}z}$
is an elementary Blaschke factor for $\lambda\in\mathbb{D}.$ Let
also $K_{B_{\sigma}}$ be the $n$-dimensional space defined by\[
K_{B_{\sigma}}=\mathcal{L}in\left(k_{\lambda_{i}}\,:\, i=1,\,...,\, n\right),\]
 where $\sigma$ is a family of distincts elements of $\mathbb{D}$,
and where $k_{\lambda}=\frac{1}{1-\overline{\lambda}z}$ is the Szegö
kernel associated to $\lambda\,.$ An obvious modification allows
to generalize the definition of $K_{B_{\sigma}}$ in the case where
the sequence $\sigma$ admits multiplicities.

Notice that using the scalar product $\left(\cdot,\,\cdot\right)_{H^{2}}$
on $H^{2},$ an equivalent description of this space is:

\[
K_{B_{\sigma}}=(B_{\sigma}H^{2})^{\perp}=H^{2}\ominus B_{\sigma}H^{2},\]
 where $H^{2}$ stands for the standard Hardy space of the unit disc
$\mathbb{D}$,\[
H^{2}=\left\{ f=\sum_{k\geq0}\hat{f}(k)z^{k}:\,\,\left\Vert f\right\Vert _{2}^{2}=\sup_{0\leq r<1}\int_{\mathbb{T}}\left|f(rz)\right|^{2}dm(z)<\infty\right\} ,\]
 $m$ being the Lebesgue normalized measure on $\mathbb{T}.$ We notice
that the case $\lambda_{1}=\lambda_{2}=\dots=\lambda_{n}=0$ gives
$K_{B_{\sigma}}=\mathcal{P}_{n-1}$. The issue of this paper is to
generalize classical Bernstein inequality (1) to the spaces $K_{B_{\sigma}}.$
Notice that every rational functions with poles outside of $\overline{\mathbb{D}}$
lies in a space $K_{B_{\sigma}}.$ It has already been proved in {[}Z1{]}
that if $r=\max_{j}\left|\lambda_{j}\right|,$ and $f\in K_{B_{\sigma}},$
then

\global\long\def\theequation{${2}$}
\begin{equation}
\left\Vert f'\right\Vert _{2}\leq\frac{5}{2}\frac{n}{1-r}\left\Vert f\right\Vert _{2}.\label{eq:}\end{equation}

In fact, Bernstein-type inequalities for rational functions were the
subject of a number of papers and monographs (see, for instance, {[}L{]},
{[}BoEr{]}, {[}DeLo{]}, {[}B1{]}, {[}B2{]}, {[}B3{]}, {[}B4{]} and
{[}B5{]}). Perhaps, the stronger and closer to ours of all known results
are due to K. Dyakonov {[}Dya1{]}. In particular, it is proved in
{[}Dya1{]} that the norm $\left\Vert D\right\Vert _{K_{B}\rightarrow H^{2}}$
of the differentiation operator $Df=f'$ on a space $K_{B}$ satisfies
the following double inequality \[
a\left\Vert B'\right\Vert _{\infty}\leq\left\Vert D\right\Vert _{K_{B}\rightarrow H^{2}}\leq A\left\Vert B'\right\Vert _{\infty},\]
 where $a=\frac{1}{36c}$, $A=\frac{36+c}{2\pi}$ and $c=2\sqrt{3\pi}$
(as one can check easily ($c$ is not precised in {[}Dya1{]})). It
implies an inequality of type $(2)$ (with a constant about $\frac{13}{2}$
instead of $\frac{5}{2}$) .

Our goal is to find an inequality for $\sup\,\left\Vert D\right\Vert _{K_{B}\rightarrow H^{2}}=C_{n,\, r}$
($\sup$ is over all $B$ with given $n={\rm deg}\, B$ and $r=\max\,_{\lambda\in\sigma}\left|\lambda\right|$),
which is asymptotically sharp as $n\rightarrow\infty$. Our result
is that there exists a limit $\lim_{n\rightarrow\infty}\frac{C_{n,\, r}}{n}=\frac{1+r}{1-r}$
for every $r,\;0\leq r<1$. Our method is different from {[}Dya1{]}
and is based on an elementary Hilbert space construction for an orthonormal
basis in $K_{B}$.

\section*{II. The result}
\begin{thm*}
\begin{flushleft}
Let $n\geq1,$ $\sigma=\left\{ \lambda_{1},\,...,\,\lambda_{n}\right\} $
be a sequence in the unit disc $\mathbb{D}$, and $B_{\sigma}$ the
finite Blaschke product $B_{\sigma}=\prod_{i=1}^{n}b_{\lambda_{i}}$,
where $b_{\lambda}=\frac{\lambda-z}{1-\overline{\lambda}z}$ is an
elementary Blaschke factor for $\lambda\in\mathbb{D}.$ Let also $K_{B_{\sigma}}$
be the $n$-dimensional subspace of $H^{2}$ defined by 
\par\end{flushleft}

\[
K_{B_{\sigma}}=(B_{\sigma}H^{2})^{\perp}=H^{2}\ominus B_{\sigma}H^{2}.\]
 Let $D$ be the operator of differentiation on $\left(K_{B_{\sigma}},\,\left\Vert \cdot\right\Vert _{2}\right):$
\[
D:\:\left(K_{B_{\sigma}},\,\left\Vert \cdot\right\Vert _{2}\right)\rightarrow\left(H^{2},\,\left\Vert \cdot\right\Vert _{2}\right)\]
 \[
f\mapsto f',\]
 where $\left\Vert f\right\Vert _{2}=\frac{1}{\sqrt{2\pi}}\left(\int_{\mathbb{T}}\left|f(\zeta)\right|^{2}d\zeta\right)^{\frac{1}{2}}.$
For $r\in[0,\,1)$ and $n\geq1$ , we set\[
C_{n,\, r}=\sup\left\{ \left\Vert D\right\Vert _{K_{B_{\sigma}}\rightarrow H^{2}}:\,1\leq{\rm card}\,\sigma\leq n,\,\left|\lambda\right|\leq r\:\forall\lambda\in\sigma\right\} .\]
 (i) If $n=1$ and $\sigma=\{\lambda\},$ we have \[
\left\Vert D\right\Vert _{K_{B_{\sigma}}\rightarrow H^{2}}=\left|\lambda\right|\left(\frac{1}{1-\left|\lambda\right|^{2}}\right)^{\frac{1}{2}}.\]
 If $n\geq2,$ \[
a(n,\, r)\frac{n}{1-r}\leq C_{n,\, r}\leq A(n,\, r)\frac{n}{1-r},\]
 where \[
a(n,\, r)\geq\frac{1}{1+r}\left(1+5r^{4}-\frac{4r^{4}}{n}-\min\left(\frac{3}{4},\,\frac{2}{n}\right)\right)^{\frac{1}{2}},\]
 and \[
A(n,\, r)\leq1+r+\frac{1}{\sqrt{n}}.\]
 (ii) Moreover, the sequence \[
\left(\frac{1}{n}C_{n,\, r}\right)_{n\geq1},\]
 is convergent and \[
\lim_{n\rightarrow\infty}\frac{1}{n}C_{n,\, r}=\frac{1+r}{1-r},\]
 for all $r\in[0,\,1)$.\end{thm*}
\begin{proof}
\textbf{We first prove (i).} We suppose that $n=1.$ In this case,
$K_{B}=\mathbb{C}e_{1}$ , where \[
e_{1}=\frac{\left(1-\left|\lambda\right|^{2}\right)^{\frac{1}{2}}}{\left(1-\overline{\lambda}z\right)},\,\left|\lambda\right|\leq r,\]
 ($e_{1}$ being of norm 1 in $H^{2}$). Calculating, \[
e_{1}'=\frac{\overline{\lambda}\left(1-\left|\lambda\right|^{2}\right)^{\frac{1}{2}}}{\left(1-\overline{\lambda}z\right)^{2}},\]
 and \[
\left\Vert e_{1}'\right\Vert _{2}=\left|\lambda\right|\left(1-\left|\lambda\right|^{2}\right)^{\frac{1}{2}}\left\Vert \frac{1}{\left(1-\overline{\lambda}z\right)^{2}}\right\Vert _{2}=\]
 \[
=\left|\lambda\right|\left(1-\left|\lambda\right|^{2}\right)^{\frac{1}{2}}\left(\sum_{k\geq0}(k+1)\left|\lambda\right|^{2k}\right)^{\frac{1}{2}}=\left|\lambda\right|\left(1-\left|\lambda\right|^{2}\right)^{\frac{1}{2}}\frac{1}{\left(1-\left|\lambda\right|^{2}\right)}=\left|\lambda\right|\left(\frac{1}{1-\left|\lambda\right|^{2}}\right)^{\frac{1}{2}},\]
 we get

\textit{\[
\left\Vert D_{\vert K_{B_{\sigma}}}\right\Vert =\left|\lambda\right|\left(\frac{1}{1-\left|\lambda\right|^{2}}\right)^{\frac{1}{2}}.\]
 }

\begin{flushleft}
Now, we suppose that $n\geq2.$ First, we prove the left-hand side
inequality. Let \[
e_{n}=\frac{\left(1-r^{2}\right)^{\frac{1}{2}}}{1-rz}b_{r}^{n-1}.\]
 Then $e_{n}\in K_{b_{r}^{n}}$ and $\left\Vert e_{n}\right\Vert _{2}=1,$
(see {[}N1{]}, Malmquist-Walsh Lemma, p.116). Moreover,\[
e_{n}'=\frac{r\left(1-r^{2}\right)^{\frac{1}{2}}}{\left(1-rz\right)^{2}}b_{r}^{n-1}+(n-1)\frac{\left(1-r^{2}\right)^{\frac{1}{2}}}{1-rz}b_{r}'b_{r}^{n-2}=\]
 \[
=-\frac{r}{\left(1-r^{2}\right)^{\frac{1}{2}}}b_{r}'b_{r}^{n-1}+(n-1)\frac{\left(1-r^{2}\right)^{\frac{1}{2}}}{1-rz}b_{r}'b_{r}^{n-2},\]
 since $b_{r}^{'}=\frac{r^{2}-1}{\left(1-rz\right)^{2}}$. Then,\[
e_{n}'=b_{r}'\left[-\frac{r}{\left(1-r^{2}\right)^{\frac{1}{2}}}b_{r}^{n-1}+(n-1)\frac{\left(1-r^{2}\right)^{\frac{1}{2}}}{1-rz}b_{r}^{n-2}\right],\]
 and \[
\left\Vert e_{n}'\right\Vert _{2}^{2}=\frac{1}{2\pi}\int_{\mathbb{T}}\left|b_{r}'(w)\right|\left|b_{r}'(w)\right|\left|-\frac{r}{\left(1-r^{2}\right)^{\frac{1}{2}}}\left(b_{r}(w)\right)^{n-1}+(n-1)\frac{\left(1-r^{2}\right)^{\frac{1}{2}}}{1-rw}\left(b_{r}(w)\right)^{n-2}\right|^{2}dm(w)=\]
 \[
=\frac{1}{2\pi}\int_{\mathbb{T}}\left|b_{r}'(w)\right|\left|b_{r}'(w)\right|\left|-\frac{r}{\left(1-r^{2}\right)^{\frac{1}{2}}}b_{r}(w)+(n-1)\frac{\left(1-r^{2}\right)^{\frac{1}{2}}}{1-rw}\right|^{2}dm(w),\]
 which gives, using the variables $u=b_{r}(w)$,\[
\left\Vert e_{n}'\right\Vert _{2}^{2}=\frac{1}{2\pi}\int_{\mathbb{T}}\left|b_{r}'\left(b_{r}(u)\right)\right|\left|-\frac{r}{\left(1-r^{2}\right)^{\frac{1}{2}}}u+(n-1)\frac{\left(1-r^{2}\right)^{\frac{1}{2}}}{1-rb_{r}(u)}\right|^{2}dm(u).\]
 But $1-rb_{r}=\frac{1-rz-r(r-z)}{1-rz}=\frac{1-r^{2}}{1-rz}$ and
$b_{r}'\circ b_{r}=\frac{r^{2}-1}{\left(1-rb_{r}\right)^{2}}=-\frac{\left(1-rz\right)^{2}}{1-r^{2}}$.
This implies \[
\left\Vert e_{n}'\right\Vert _{2}^{2}=\frac{1}{2\pi}\int_{\mathbb{T}}\left|\frac{\left(1-ru\right)^{2}}{1-r^{2}}\right|\left|-\frac{r}{\left(1-r^{2}\right)^{\frac{1}{2}}}u+(n-1)\frac{\left(1-r^{2}\right)^{\frac{1}{2}}}{1-r^{2}}(1-ru)\right|^{2}dm(u)=\]
 \[
=\frac{1}{\left(1-r^{2}\right)^{2}}\frac{1}{2\pi}\int_{\mathbb{T}}\left|\left(1-ru\right)\left(-ru+(n-1)(1-ru)\right)\right|^{2}dm(u).\]
 Without loss of generality we can replace $r$ by $-r,$ which gives\[
\left\Vert e_{n}'\right\Vert _{2}=\frac{1}{1-r^{2}}\left\Vert \varphi_{n}\right\Vert _{2},\]
 where $\varphi_{n}=\left(1+rz\right)\left(rz+(n-1)(1+rz)\right).$
Expanding, we get \[
\varphi_{n}=(1+rz)(nrz+(n-1))=\]
 \[
=nrz+(n-1)+nr^{2}z^{2}+(n-1)r^{}z=\]
 \[
=(n-1)+(nr+(n-1)r^{})z+nr^{2}z^{2},\]
 and \[
\left\Vert e_{n}'\right\Vert _{2}^{2}=\frac{1}{\left(1-r^{2}\right)^{2}}\left((n-1)^{2}+(2n-1)^{2}r^{2}+n^{2}r^{4}\right)=\]
 \[
=\frac{n^{2}}{\left(1-r^{2}\right)^{2}}\left(1+4r^{2}+r^{4}-\frac{2}{n}-\frac{4r^{2}}{n}+\frac{1}{n^{2}}+\frac{r^{2}}{n^{2}}\right)=\]
 \[
=\left(\frac{n}{1-r}\right)^{2}\left(\frac{1}{1+r}\right)^{2}\left(1+4r^{2}+r^{4}-\frac{2}{n}-\frac{4r^{2}}{n}+\frac{1+r^{2}}{n^{2}}\right)=\]
 \[
=\left(\frac{n}{1-r}\right)^{2}\left(\frac{1}{1+r}\right)^{2}\left(1+4r^{2}+r^{4}-5r^{4}+5r^{4}-\frac{4r^{4}}{n}+\frac{4r^{4}}{n}-\frac{4r^{2}}{n}-\frac{2}{n}+\frac{1+r^{2}}{n^{2}}\right)=\]
 \[
=\left(\frac{n}{1-r}\right)^{2}\left(\frac{1}{1+r}\right)^{2}\left(4r^{2}(1-r^{2})-\frac{4r^{2}}{n}(1-r^{2})+\frac{1+r^{2}}{n^{2}}+1+5r^{4}-\frac{4r^{4}}{n}-\frac{2}{n}\right)=\]
 \[
=\left(\frac{n}{1-r}\right)^{2}\left(\frac{1}{1+r}\right)^{2}\left(4r^{2}(1-r^{2})(1-\frac{1}{n})+\frac{1+r^{2}}{n^{2}}+1+5r^{4}-\frac{4r^{4}}{n}-\frac{2}{n}\right)\geq\]
 \[
\geq\left(\frac{n}{1-r}\right)^{2}\left(\frac{1}{1+r}\right)^{2}\left\{ \begin{array}{c}
1+5r^{4}-\frac{4r^{4}}{n}-\frac{2}{n}\; if\; n>2\\
\frac{1}{4}+1+5r^{4}-\frac{4r^{4}}{2}-\frac{2}{2}\; if\; n=2\end{array}\right.\geq\]
 \[
\geq\left(\frac{n}{1-r}\right)^{2}\left(\frac{1}{1+r}\right)^{2}\left(1+5r^{4}-\frac{4r^{4}}{n}-\min\left(\frac{3}{4},\,\frac{2}{n}\right)\right),\]
 and \textit{\[
a(n,\, r)\geq\frac{1}{1+r}\left(1+5r^{4}-\frac{4r^{4}}{n}-\min\left(\frac{3}{4},\,\frac{2}{n}\right)\right)^{\frac{1}{2}},\]
 } 
\par\end{flushleft}

\begin{flushleft}
which completes the proof of the left hand side inequality . 
\par\end{flushleft}

\begin{flushleft}
We show now the right hand side one. Let $\sigma$ be a sequence in
$\mathbb{D}$ such that $1\leq{\rm card}\,\sigma\leq n,\,\left|\lambda\right|\leq r\:\forall\lambda\in\sigma$.
Using {[}Z1{]}, Proposition 4.1, we have \[
\left\Vert D\right\Vert _{K_{B_{\sigma}}\rightarrow H^{2}}\leq\frac{1}{1-r}+\frac{1+r}{1-r}(n-1)+\frac{1}{1-r}\sqrt{n-2}=\]
 \[
=\frac{1}{1-r}\left(1+(1+r)(n-1)+\sqrt{n-2}\right)=\]
 \[
=\frac{1}{1-r}\left(n(1+r)-r+\sqrt{n-2}\right)=\frac{n}{1-r}\left(1+r-\frac{r}{n}+\sqrt{\frac{1}{n}-\frac{2}{n^{2}}}\right)=\]
 \[
\leq\frac{n}{1-r}\left(1+r+\sqrt{\frac{1}{n}}\right),\]
 which gives the result. 
\par\end{flushleft}

Now, we prove (ii). \textbf{Step 1. }We first prove the right-hand
side inequality:\[
\limsup_{n\rightarrow\infty}\frac{1}{n}C_{n,\, r}\leq\frac{1+r}{1-r},\]
 which becomes obvious since \[
\left\Vert D\right\Vert _{K_{B_{\sigma}}\rightarrow H^{2}}\leq\frac{n}{1-r}\left(1+r+\sqrt{\frac{1}{n}}\right)\,.\]

\textbf{Step 2.} We now prove the left-hand side inequality:\[
\liminf_{n\rightarrow\infty}\frac{1}{n}C_{n,\, r}\geq\frac{1+r}{1-r}.\]
 More precisely, we show that\[
\liminf_{n\rightarrow\infty}\frac{1}{n}\left\Vert D\right\Vert _{K_{b_{r}^{n}}\rightarrow H^{2}}\geq\frac{1+r}{1-r}.\]
 Let $f\in K_{b_{r}^{n}}.$ Then, \[
f'=\left(f,\, e_{1}\right)_{H^{2}}\frac{r}{\left(1-rz\right)}e_{1}+\sum_{k=2}^{n}(k-1)\left(f,\, e_{k}\right)_{H^{2}}\frac{b_{r}'}{b_{r}}e_{k}+r\sum_{k=2}^{n}\left(f,\, e_{k}\right)_{H^{2}}\frac{1}{\left(1-rz\right)}e_{k}=\]
 \[
=r\sum_{k=1}^{n}\left(f,\, e_{k}\right)_{H^{2}}\frac{1}{\left(1-rz\right)}e_{k}+\frac{1-r^{2}}{(1-rz)(z-r)}\sum_{k=2}^{n}(k-1)\left(f,\, e_{k}\right)_{H^{2}}e_{k}=\]
 \[
=\frac{r\left(1-r^{2}\right)^{\frac{1}{2}}}{\left(1-rz\right)^{2}}\sum_{k=1}^{n}\left(f,\, e_{k}\right)_{H^{2}}b_{r}^{k-1}+\frac{\left(1-r^{2}\right)^{\frac{3}{2}}}{(1-rz)^{2}(z-r)}\sum_{k=2}^{n}(k-1)\left(f,\, e_{k}\right)_{H^{2}}b_{r}^{k-1},\]
 which gives

\global\long\def\theequation{${3}$}
\begin{equation}
f'=-b_{r}'\left[\frac{r}{\left(1-r^{2}\right)^{\frac{1}{2}}}\sum_{k=1}^{n}\left(f,\, e_{k}\right)_{H^{2}}b_{r}^{k-1}+\frac{\left(1-r^{2}\right)^{\frac{1}{2}}}{z-r}\sum_{k=2}^{n}(k-1)\left(f,\, e_{k}\right)_{H^{2}}b_{r}^{k-1}\right].\label{eq:}\end{equation}
 Now using the change of variables $v=b_{r}(u),$ we get\[
\left\Vert f'\right\Vert _{2}^{2}=\int_{\mathbb{T}}\left|b_{r}'(u)\right|\left|b_{r}'(u)\right|\left|\frac{r}{\left(1-r^{2}\right)^{\frac{1}{2}}}\sum_{k=1}^{n}\left(f,\, e_{k}\right)_{H^{2}}b_{r}^{k-1}+\frac{\left(1-r^{2}\right)^{\frac{1}{2}}}{u-r}\sum_{k=2}^{n}(k-1)\left(f,\, e_{k}\right)_{H^{2}}b_{r}^{k-1}\right|^{2}du=\]
 \[
=\int_{\mathbb{T}}\left|b_{r}'(b_{r}(v))\right|\left|\frac{r}{\left(1-r^{2}\right)^{\frac{1}{2}}}\sum_{k=1}^{n}\left(f,\, e_{k}\right)_{H^{2}}v^{k-1}+\frac{\left(1-r^{2}\right)^{\frac{1}{2}}}{b_{r}(v)-r}\sum_{k=2}^{n}(k-1)\left(f,\, e_{k}\right)_{H^{2}}v^{k-1}\right|^{2}dv.\]
 But \[
b_{r}-r=\frac{r-z-r(1-rz)}{1-rz}=\frac{z(r^{2}-1)}{1-rz},\]
 and\[
b_{r}'\circ b_{r}=\frac{r^{2}-1}{\left(1-rb_{r}\right)^{2}}=-\frac{\left(1-rz\right)^{2}}{1-r^{2}},\]
 which gives\[
\left\Vert f'\right\Vert _{2}^{2}=\]
 \[
=\frac{1}{1-r^{2}}\int_{\mathbb{T}}\left|\left(1-rv\right)^{2}\right|\left|\frac{r}{\left(1-r^{2}\right)^{\frac{1}{2}}}\sum_{k=1}^{n}\left(f,\, e_{k}\right)_{H^{2}}v^{k-1}+\frac{\left(1-r^{2}\right)^{\frac{1}{2}}}{v(r^{2}-1)}(1-rv)\sum_{k=2}^{n}(k-1)\left(f,\, e_{k}\right)_{H^{2}}v^{k-1}\right|^{2}dv=\]
 \[
=\frac{1}{\left(1-r^{2}\right)^{2}}\int_{\mathbb{T}}\left|\left(1-rv\right)^{2}\right|\left|r\sum_{k=1}^{n}\left(f,\, e_{k}\right)_{H^{2}}v^{k-1}-(1-rv)\sum_{k=2}^{n}(k-1)\left(f,\, e_{k}\right)_{H^{2}}v^{k-2}\right|^{2}dv,\]
 and

\global\long\def\theequation{${4}$}
\begin{equation}
\left\Vert f'\right\Vert _{2}^{2}=\frac{1}{\left(1-r^{2}\right)^{2}}\int_{\mathbb{T}}\left|r\left(1-rv\right)\sum_{k=0}^{n-1}\left(f,\, e_{k+1}\right)_{H^{2}}v^{k}-(1-rv)^{2}\sum_{k=0}^{n-2}(k+1)\left(f,\, e_{k+2}\right)_{H^{2}}v^{k}\right|^{2}dv.\label{eq:}\end{equation}
 In particular,

\global\long\def\theequation{${5}$}
\begin{equation}
\frac{1}{n\left\Vert f\right\Vert _{2}}\left[\left\Vert (1-rv)^{2}\sum_{k=0}^{n-2}(k+1)\left(f,\, e_{k+2}\right)_{H^{2}}v^{k}\right\Vert _{2}+\left\Vert r\left(1-rv\right)\sum_{k=0}^{n-1}\left(f,\, e_{k+1}\right)_{H^{2}}v^{k}\right\Vert _{2}\right]\geq\label{eq:}\end{equation}

\[
\geq\frac{1-r^{2}}{n}\frac{\left\Vert f'\right\Vert _{2}}{\left\Vert f\right\Vert _{2}}\geq\]
 \[
\frac{1}{n\left\Vert f\right\Vert _{2}}\left[\left\Vert (1-rv)^{2}\sum_{k=0}^{n-2}(k+1)\left(f,\, e_{k+2}\right)_{H^{2}}v^{k}\right\Vert _{2}-\left\Vert r\left(1-rv\right)\sum_{k=0}^{n-1}\left(f,\, e_{k+1}\right)_{H^{2}}v^{k}\right\Vert _{2}\right].\]
 Now, we notice that on one hand

\global\long\def\theequation{${6}$}
\begin{equation}
\left\Vert r\left(1-rv\right)\sum_{k=0}^{n-1}\left(f,\, e_{k+1}\right)_{H^{2}}v^{k}\right\Vert _{2}\leq r(1+r)\left(\sum_{k=0}^{n-1}\left|\left(f,\, e_{k+1}\right)_{H^{2}}\right|^{2}\right)^{1/2}\leq r(1+r)\left\Vert f\right\Vert _{2},\label{eq:}\end{equation}
 and on the other hand, \[
(1-rv)^{2}\sum_{k=0}^{n-2}(k+1)\left(f,\, e_{k+2}\right)_{H^{2}}v^{k}=\]
 \[
=(1-2rv+r^{2}v^{2})\sum_{k=0}^{n-2}(k+1)\left(f,\, e_{k+2}\right)_{H^{2}}v^{k}=\]
 \[
=\sum_{k=0}^{n-2}(k+1)\left(f,\, e_{k+2}\right)_{H^{2}}v^{k}-2r\sum_{k=0}^{n-2}(k+1)\left(f,\, e_{k+2}\right)_{H^{2}}v^{k+1}+r^{2}\sum_{k=0}^{n-2}(k+1)\left(f,\, e_{k+2}\right)_{H^{2}}v^{k+2}=\]
 \[
=\sum_{k=0}^{n-2}(k+1)\left(f,\, e_{k+2}\right)_{H^{2}}v^{k}-2r\sum_{k=1}^{n-1}k\left(f,\, e_{k+1}\right)_{H^{2}}v^{k}+r^{2}\sum_{k=2}^{n}(k-1)\left(f,\, e_{k}\right)_{H^{2}}v^{k}=\]
 \[
=\left(f,\, e_{2}\right)_{H^{2}}+2\left(f,\, e_{3}\right)_{H^{2}}v+\sum_{k=2}^{n-2}\left[(k+1)\left(f,\, e_{k+2}\right)_{H^{2}}-2rk\left(f,\, e_{k+1}\right)_{H^{2}}+r^{2}(k-1)\left(f,\, e_{k}\right)_{H^{2}}\right]v^{k}+\]
 \[
-2r\left[\left(f,\, e_{2}\right)_{H^{2}}v+(n-1)\left(f,\, e_{n}\right)_{H^{2}}v^{n-1}\right]+r^{2}\left[(n-2)\left(f,\, e_{n-1}\right)_{H^{2}}v^{n-1}+(n-1)\left(f,\, e_{n}\right)_{H^{2}}v^{n}\right],\]
 which gives

\global\long\def\theequation{${7}$}
\begin{equation}
(1-rv)^{2}\sum_{k=0}^{n-2}(k+1)\left(f,\, e_{k+2}\right)_{H^{2}}v^{k}=\left(f,\, e_{2}\right)_{H^{2}}+2\left[\left(f,\, e_{3}\right)_{H^{2}}-r\left(f,\, e_{2}\right)_{H^{2}}\right]v+\label{eq:}\end{equation}

\[
+\sum_{k=2}^{n-2}\left[(k+1)\left(f,\, e_{k+2}\right)_{H^{2}}-2rk\left(f,\, e_{k+1}\right)_{H^{2}}+r^{2}(k-1)\left(f,\, e_{k}\right)_{H^{2}}\right]v^{k}+\]
 \[
+\left[r^{2}(n-2)\left(f,\, e_{n-1}\right)_{H^{2}}-2r(n-1)\left(f,\, e_{n}\right)_{H^{2}}\right]v^{n-1}+r^{2}(n-1)\left(f,\, e_{n}\right)_{H^{2}}v^{n}.\]
 Now, let $s=(s_{n})_{n}$ be a sequence of even integers such that\[
\lim_{n\rightarrow\infty}s_{n}=\infty\:\mbox{and}\; s_{n}=o(n)\;{\rm as}\; n\rightarrow\infty.\]
 Then we consider the following function $f$ in $K_{b_{r}^{n}}$:
\[
f=e_{n}-e_{n-1}+e_{n-2}-e_{n-3}+...+(-1)^{k}e_{n-k}+...+e_{n-s}-e_{n-s-1}+e_{n-s-2}=\]
 \[
=\sum_{k=0}^{s+2}(-1)^{k}e_{n-k}.\]
 Using (6) on one hand, we get

\global\long\def\theequation{${8}$}
\begin{equation}
\lim_{n\rightarrow\infty}\frac{1}{n\left\Vert f\right\Vert _{2}}\left\Vert r\left(1-rv\right)\sum_{k=0}^{n-1}\left(f,\, e_{k+1}\right)_{H^{2}}v^{k}\right\Vert _{2}=0,\label{eq:}\end{equation}
 and applying (7) on the other hand, we obtain\[
\left\Vert (1-rv)^{2}\sum_{k=0}^{n-2}(k+1)\left(f,\, e_{k+2}\right)_{H^{2}}v^{k}\right\Vert _{2}^{2}=\left|\left(f,\, e_{2}\right)_{H^{2}}\right|^{2}+4\left|\left(f,\, e_{3}\right)_{H^{2}}-r\left(f,\, e_{2}\right)_{H^{2}}\right|^{2}+\]
 \[
+\left|r^{2}(n-2)\left(f,\, e_{n-1}\right)_{H^{2}}-2r(n-1)\left(f,\, e_{n}\right)_{H^{2}}\right|^{2}+r^{4}(n-1)^{2}\left|\left(f,\, e_{n}\right)_{H^{2}}\right|^{2}+\]
 \[
+\sum_{k=2}^{n-2}\left|(k+1)\left(f,\, e_{k+2}\right)_{H^{2}}-2rk\left(f,\, e_{k+1}\right)_{H^{2}}+r^{2}(k-1)\left(f,\, e_{k}\right)_{H^{2}}\right|^{2},\]
 which gives \[
\left\Vert (1-rv)^{2}\sum_{k=0}^{n-2}(k+1)\left(f,\, e_{k+2}\right)_{H^{2}}v^{k}\right\Vert _{2}^{2}=\]
 \[
=\left|r^{2}(n-2)+2r(n-1)\right|^{2}+r^{4}(n-1)^{2}+\]
 \[
+\sum_{l=2}^{n-2}\left|(n-l+1)\left(f,\, e_{n-l+2}\right)_{H^{2}}-2r(n-l)\left(f,\, e_{n-l+1}\right)_{H^{2}}+r^{2}(n-l-1)\left(f,\, e_{n-l}\right)_{H^{2}}\right|^{2},\]
 setting the change of index $l=n-k$ in the last sum. This finally
gives\[
\left\Vert (1-rv)^{2}\sum_{k=0}^{n-2}(k+1)\left(f,\, e_{k+2}\right)_{H^{2}}v^{k}\right\Vert _{2}^{2}=\]
 \[
=\left|r^{2}(n-2)+2r(n-1)\right|^{2}+r^{4}(n-1)^{2}+\]
 \[
+\sum_{l=2}^{s+1}\left|(n-l+1)+2r(n-l)+r^{2}(n-l-1)\right|^{2}+\]
 \[
+\left|(n-s-1)+2r(n-s-2)\right|^{2}+\left|n-s-2\right|^{2}.\]
 And\[
\left\Vert (1-rv)^{2}\sum_{k=0}^{n-2}(k+1)\left(f,\, e_{k+2}\right)_{H^{2}}v^{k}\right\Vert _{2}^{2}\geq\]
 \[
\geq\left|r^{2}(n-2)+2r(n-1)\right|^{2}+r^{4}(n-1)^{2}+\]
 \[
+s\left|(n-s)+2r(n-s-1)+r^{2}(n-s-2)\right|^{2}+\]
 \[
+\left|(n-s-1)+2r(n-s-2)\right|^{2}+\left|n-s-2\right|^{2}.\]
 In particular,

\global\long\def\theequation{${9}$}
\begin{equation}
\left\Vert (1-rv)^{2}\sum_{k=0}^{n-2}(k+1)\left(f,\, e_{k+2}\right)_{H^{2}}v^{k}\right\Vert _{2}^{2}\geq s\left|(n-s)+2r(n-s-1)+r^{2}(n-s-2)\right|^{2}.\label{eq:}\end{equation}
 Passing after to the limit as $n\rightarrow\infty$ in (5), we obtain
(using (8))

\global\long\def\theequation{${10}$}
\begin{equation}
\frac{1}{1+r}\liminf_{n\rightarrow\infty}\frac{1}{n\left\Vert f\right\Vert _{2}}\left\Vert (1-rv)^{2}\sum_{k=0}^{n-2}(k+1)\left(f,\, e_{k+2}\right)_{H^{2}}v^{k}\right\Vert _{2}\geq\label{eq:}\end{equation}
 \[
\geq\liminf_{n\rightarrow\infty}\frac{1-r}{n}\frac{\left\Vert f'\right\Vert _{2}}{\left\Vert f\right\Vert _{2}}\geq\]
 \[
\frac{1}{1+r}\liminf_{n\rightarrow\infty}\frac{1}{n\left\Vert f\right\Vert _{2}}\left\Vert (1-rv)^{2}\sum_{k=0}^{n-2}(k+1)\left(f,\, e_{k+2}\right)_{H^{2}}v^{k}\right\Vert _{2}.\]
 This gives

\global\long\def\theequation{${11}$}
\begin{equation}
\liminf_{n\rightarrow\infty}\frac{1-r}{n}\frac{\left\Vert f'\right\Vert _{2}}{\left\Vert f\right\Vert _{2}}=\frac{1}{1+r}\liminf_{n\rightarrow\infty}\frac{1}{n\left\Vert f\right\Vert _{2}}\left\Vert (1-rv)^{2}\sum_{k=0}^{n-2}(k+1)\left(f,\, e_{k+2}\right)_{H^{2}}v^{k}\right\Vert _{2}.\label{eq:}\end{equation}
 Now, since \[
\left\Vert f\right\Vert _{2}^{2}=s_{n}+3,\]
 using (9) we obtain\[
\liminf_{n\rightarrow\infty}\frac{1}{n^{2}\left\Vert f\right\Vert _{2}^{2}}\left\Vert (1-rv)^{2}\sum_{k=0}^{n-2}(k+1)\left(f,\, e_{k+2}\right)_{H^{2}}v^{k}\right\Vert _{2}^{2}\geq\]
 \[
\geq\liminf_{n\rightarrow\infty}\frac{1}{n^{2}\left\Vert f\right\Vert _{2}^{2}}(\left\Vert f\right\Vert _{2}^{2}-3)\left|(n-s)+2r(n-s-1)+r^{2}(n-s-2)\right|^{2}.\]
 Since\[
\lim_{n\rightarrow\infty}\frac{3}{n^{2}s_{n}^{2}}\left|(n-s)+2r(n-s-1)+r^{2}(n-s-2)\right|^{2}=0,\]
 we get \[
\liminf_{n\rightarrow\infty}\frac{1}{n^{2}\left\Vert f\right\Vert _{2}^{2}}\left\Vert (1-rv)^{2}\sum_{k=0}^{n-2}(k+1)\left(f,\, e_{k+2}\right)_{H^{2}}v^{k}\right\Vert _{2}^{2}\geq\]
 \[
\geq\liminf_{n\rightarrow\infty}\frac{1}{n^{2}s_{n}^{2}}s_{n}^{2}\left|(n-s_{n})+2r(n-s_{n}-1)+r^{2}(n-s_{n}-2)\right|^{2}=\]
 \[
=\lim_{n\rightarrow\infty}\frac{1}{n^{2}}\left|(n-s_{n})+2r(n-s_{n}-1)+r^{2}(n-s_{n}-2)\right|^{2}=\]
 \[
=\lim_{n\rightarrow\infty}\frac{1}{n^{2}}\left|n+2rn+r^{2}n\right|^{2}=(1+r)^{4}.\]
 We can now conclude that\[
\liminf_{n\rightarrow\infty}\frac{1-r}{n}\left\Vert D\right\Vert _{K_{b_{r}^{n}}\rightarrow H^{2}}\geq\liminf_{n\rightarrow\infty}\frac{1-r}{n}\frac{\left\Vert f'\right\Vert _{2}}{\left\Vert f\right\Vert _{2}}=\]
 \[
=\frac{1}{1+r}\liminf_{n\rightarrow\infty}\frac{1}{n\left\Vert f\right\Vert _{2}}\left\Vert (1-rv)^{2}\sum_{k=0}^{n-2}(k+1)\left(f,\, e_{k+2}\right)_{H^{2}}v^{k}\right\Vert _{2}\geq\frac{(1+r)^{2}}{1+r}=1+r.\]
 \textbf{Step 3. Conclusion.} Using both \textbf{Step 1 }and\textbf{
Step 2}, we get\textbf{ }\[
\limsup_{n\rightarrow\infty}\frac{1-r}{n}C_{n,\, r}=\liminf_{n\rightarrow\infty}\frac{1-r}{n}C_{n,\, r}=1+r,\]
 which means that the sequence $\left(\frac{1}{n}C_{n,\, r}\right)_{n\geq1}$
is convergent and \[
\lim_{n\rightarrow\infty}\frac{1}{n}C_{n,\, r}=\frac{1+r}{1-r}.\]

\end{proof}
\begin{flushleft}
\textbf{Comments} 
\par\end{flushleft}

(a) Bernstein-type inequalities for $K_{B}$ appeared as early as
in 1991. There, the boundedness of $D\,:\left(K_{B},\,\left\Vert \cdot\right\Vert _{H^{p}}\right)\rightarrow\left(H^{p},\,\left\Vert \cdot\right\Vert _{H^{p}}\right)$
was covered for the full range $1\le p\le\infty$. In {[}Dya1{]},
the chief concern of K. Dyakonov was compactness (plus a new, simpler,
proof of boundedness). Now, using both {[}BoEr{]} Th. 7.1.7 p. 324
, (or equivalently M. Levin's inequality {[}L{]}) and complex interpolation,
we could recover the result of K. Dyakonov for $H^{p}$ spaces, $2\leq p\leq\infty$
and our method could give a better numerical constant $c_{p}$ in
the inequality \[
\left\Vert f'\right\Vert _{H^{p}}\leq c_{p}\left\Vert B'\right\Vert _{\infty}\left\Vert f\right\Vert _{H^{p}}.\]
 The case $1\leq p\leq2$ can be treated using the partial result
of K. Dyakonov $(p=1)$ and still complex interpolation.

(b) In the same spirit, it is also possible to generalize the above
Bernstein-type inequality to the same class of rational functions
$f$ in $\mathbb{D}$, replacing the Hardy space $H^{2}$ by Besov
spaces $B_{2,\,2}^{s}$, $s\in\mathbb{R},$ of all holomorphic functions
$f=\sum_{k\geq0}\hat{f}(k)z^{k}$ in $\mathbb{D}$ satisfying \[
\left\Vert f\right\Vert _{B_{2,\,2}^{s}}:=\left(\sum_{k\geq0}(k+1)^{2s}\left|\hat{f}(k)\right|^{2}\right)^{\frac{1}{2}}<\infty.\]
 The same spaces are also known as Dirichlet-Bergman spaces. (In particular,
the classical Bergman space corresponds to $s=-\frac{1}{2}$ and the
classical Dirichlet space corresponds to $s=\frac{1}{2}$). Using
the above approach, one can prove the sharpness of the growth order
$\frac{n}{1-r}$ in the corresponding Bernstein-type inequality

\global\long\def\theequation{${3}$}

\begin{equation}
\left\Vert f'\right\Vert _{B_{2,\,2}^{s}}\leq c_{s}\frac{n}{1-r}\left\Vert f\right\Vert _{B_{2,\,2}^{s}},\label{eq:}\end{equation}
 (at least for integers values of $s$).

(c) One can also prove an inequality

\global\long\def\theequation{${4}$}

\begin{equation}
\left\Vert f\right\Vert _{B_{2,\,2}^{s}}\leq c_{s}'\left(\frac{n}{1-r}\right)^{s}\left\Vert f\right\Vert _{H^{2}},\label{eq:}\end{equation}
 for $s\geq0$ and the same class of functions (essentially, this
inequality can be found in {[}Dya2{]}), and show the sharpness of
the growth order $\left(\frac{n}{1-r}\right)^{s}$ (at least for integers
values of $s$). An application of this inequality lies in constrained
$H^{\infty}$ interpolation in weighted Hardy and Bergman spaces,
see {[}Z1{]} and {[}Z2{]} for details.

Notice that already E. M. Dyn'kin (in {[}Dyn{]}), and A. A. Pekarskii
(in {[}Pe1{]}, {[}Pe2{]} and {[}PeSt{]}), studied Bernstein-type inequalities
for rational functions in Besov and Sobolev spaces. In particular,
they applied such inequalities to inverse theorems of rational approximation.
Our approach is different and more constructive. We are able to obtain
uniform bounds depending on the geometry of poles of order $n$, which
allows us to obtain estimates which are asymptotically sharp.

Also, in paper {[}Dya3{]} of K. Dyakonov (see Sections 10, 11 at the
end), there are Bernstein-type inequalities involving Besov and Sobolev
spaces that contain, as special cases, the earlier version from ,
Pekarskii's inequalities for rational functions, and much more. K.
Dyakonov used those Bernstein-type inequalities to \textquotedbl{}interpolate\textquotedbl{},
in a sense, between the polynomial and rational inverse approximation
theorems (in response to a question raised by E. M. Dyn'kin). Finally,
he has recently studied the \textquotedbl{}reverse Bernstein inequality\textquotedbl{}
in $K_{B}$; this is done in {[}Dya3{]}.

(d) The above comments can lead to wonder what happens if we replace
Besov spaces $B_{2,\,2}^{s}$ by other Banach spaces, for example
by $W$, the Wiener algebra of absolutely convergent Taylor series.
In this case, we obtain

\global\long\def\theequation{${5}$}

\begin{equation}
\left\Vert f\right\Vert _{W}\leq c(n,r)\left\Vert f\right\Vert _{H^{2}}\label{eq:}\end{equation}
 where $c(n,\, r)\leq c\left(\frac{n^{2}}{1-r}\right)^{\frac{1}{2}}$
and $c$ is a numerical constant. We suspect that $\left(\frac{n^{2}}{1-r}\right)^{\frac{1}{2}}$
is the right growth order of $c(n,\, r)$. An application of this
inequality to an estimate of the norm of the resolvent of an $n\times n$
power-bounded matrix $T$ on a Banach space is given in {[}Z3{]}.
Inequality $(5),$ above, is deeply linked with the inequality

\global\long\def\theequation{${6}$}

\begin{equation}
\left\Vert f'\right\Vert _{H^{1}}\leq\gamma n\left\Vert f\right\Vert _{H^{\infty}},\label{eq:}\end{equation}
 through Hardy's inequality : \[
\left\Vert f\right\Vert _{W}\leq\pi\left\Vert f'\right\Vert _{H^{1}}+\left|f(0)\right|,\]
 for all $f\in W\,,$ (see {[}N2{]} p. 370 8.7.4 -(c)).

Inequality (6) is (shown and) used by R. J. LeVeque and L. N. Trefethen
in {[}LeTr{]} with $\gamma=2$, and later by M. N. Spijker in {[}Sp{]}
with $\gamma=1$ (an improvement) so as to apply it to the Kreiss
Matrix Theorem in which the power boundedness of $n\times n$ matrices
is related to a resolvent condition on these matrices.

\pagebreak{}

\begin{flushleft}
\textbf{Acknowledgement.} 
\par\end{flushleft}

I would like to thank Professor Nikolai Nikolski for his invaluable
help and his precious advices. I also would like to thank Professor
Alexander Borichev for many helpful discussions.

\vspace{0.4cm}

\noun{CMI-LATP, UMR 6632, Université de Provence, 39, rue F.-Joliot-Curie,
13453 Marseille cedex 13, France}

\textit{E-mail address} : rzarouf@cmi.univ-mrs.fr 
\end{document}